\documentclass[11pt]{amsart}

\usepackage{graphicx}
\graphicspath{ {images/} }
\usepackage{amssymb,verbatim,amscd,amsfonts,amsmath,amsthm,enumerate}
\usepackage[all]{xy}
\usepackage{subfig}
\usepackage{tikz}
\usetikzlibrary{shapes,arrows,shadows}
\usetikzlibrary{decorations.markings}
\usepackage{color}
\usepackage[normalem]{ulem}

\usepackage{hyperref}
\usepackage{wrapfig}
\usetikzlibrary{arrows}
\usepackage{pinlabel}

\long\def\symbolfootnote[#1]#2{\begingroup%
\def\thefootnote{\fnsymbol{footnote}}\footnote[#1]{#2}\endgroup}

\newtheorem{theorem}{Theorem}[section]

\newtheorem{proposition}[theorem]{Proposition}
\newtheorem{corollary}[theorem]{Corollary}
\newtheorem{lemma}[theorem]{Lemma}

\newtheorem{question}[theorem]{Question}

\theoremstyle{definition}

\newtheorem{remark}[theorem]{Remark}

\newtheorem{definition}[theorem]{Definition}

\newtheorem{example}[theorem]{Example}

\newtheorem*{namedtheorem}{\theoremname}
\newcommand{\theoremname}{testing}

\newcommand{\Ends}{{\text{Ends}}}

\newcommand{\Ac}{\mathcal{A}}
\newcommand{\AC}{\mathcal{AC}}
\newcommand{\Cc}{\mathcal{C}}

\newcommand{\G}{\mathcal{G}}

\def\Mod{\operatorname{Mod}}

\def\Nonsep{\operatorname{NonSep}}
\def\Outer{\operatorname{Outer}}
\def\Ends{\operatorname{Ends}}

\def\diam{\operatorname{diam}}

\begin{document}
\title[Graphs for infinite-type surfaces]{On the geometry of graphs associated to infinite-type surfaces}
\author{Javier Aramayona   \& Ferr\'an Valdez}
\thanks{The first author was partially funded by grants RYC-2013-13008 and MTM2015-67781. The second author was  supported by PAPIIT projects IN100115, IN103411 and IB100212.}

%

\begin{abstract}
Consider a connected orientable surface $S$ of infinite topological type, i.e. with infinitely-generated fundamental group.

Our main purpose is to give a description of the geometric structure of an arbitrary subgraph of the arc graph of $S$, subject to some rather general conditions. As special cases, we recover results of J. Bavard \cite{Bavard} and Aramayona-Fossas-Parlier \cite{AFP}. 

In the second part of the paper, we obtain a number of results on the geometry of connected, $\Mod(S)$-invariant subgraphs of the curve graph 
of $S$, in the case when the space of ends of $S$ is homeomorphic to a Cantor set.

\end{abstract}

\maketitle

\section{Introduction} 

 There has been a recent surge of activity around mapping class groups of infinite-type surfaces, i.e. with infinitely-generated fundamental group. The motivation for studying these groups stems from several places, as we now briefly describe.  

First, infinite-type surfaces appear as inverse limits of surfaces of finite type. In particular, infinite-type mapping class groups  are useful in the study of asymptotic and/or stable properties of their finite-type counterparts. This is the approach taken  by Funar-Kapoudjian \cite{FK1}, where the authors identify the  homology of an
infinite-type mapping class group with the stable homology of the mapping class groups of its finite-type subsurfaces.

In a related direction, a number of well-known groups appear as subgroups of the mapping class group of infinite-type surfaces. For instance,  Funar-Kapoudjian \cite{FK2} realized Thompson's group $T$ as a topologically-defined subgroup of the mapping class group of a certain infinitely-punctured sphere. 

A third piece of motivation for studying mapping class groups of infinite-type surfaces comes from dynamics, as explained by D. Calegari in \cite{Calegari}. More concretely, let $S$ be a closed surface, $P\subset S$ a finite subset, and consider the group $\mathrm{Homeo}(S, P)$ of those homeomorphisms of $S$ that preserve $P$ setwise. Let $G < \mathrm{Homeo}(S, P)$ be a subgroup that acts freely on $S-P$. Then $G$ admits a natural homeomorphism to $\Mod(S - K, P)$, where $K$ is either a finite set or a Cantor set. See \cite{Calegari} for more details. 

%
%
%
%
%
%

\subsection{Combinatorial models} A large number of problems about mapping class groups of finite-type surfaces may be understood through the various simplicial complexes built from curves and/or arcs on surfaces. Notable examples of these are the {\em curve graph} $\Cc(S)$ and the {\em arc graph} $\Ac(S)$; see Section \ref{sec:defcurves} for definitions. When $S$ has finite  type, a useful feature of these complexes is that, with respect to their standard path-metric, they are hyperbolic spaces of infinite diameter; see \cite{MS} and \cite{MM1}, respectively. 

In sharp contrast, in the case of an infinite-type surface these complexes often have finite diameter; see Section \ref{sec:defcurves}. This obstacle was first overcome by J. Bavard \cite{Bavard} in the particular case when $S$ is a
sphere minus a Cantor set and an isolated point. Indeed, she proved that a certain subgraph of the arc graph is hyperbolic and has infinite diameter, and used this to construct non-trivial quasi-morphisms from the mapping class group of the plane minus a Cantor set. Subsequently, Aramayona-Fossas-Parlier \cite{AFP} have produced similar graphs for arbitrary surfaces, subject to certain conditions on the set of punctures of $S$. However, the definition of these subgraphs is surprisingly subtle, and small variations in the definition may produce graphs that have finite diameter or are not hyperbolic.

\subsubsection{Arc graphs} Our main goal is to give a unified description of the possible geometric structures of an arbitrary subgraph of the arc graph of an infinite-type surface, subject to some rather natural conditions on the given graph. First, we will require that it be {\em sufficiently invariant}, that is invariant under $\Mod(S,P)$, for some (possibly empty) finite set $P$ of punctures. In addition, we will assume that every such graph satisfies the {\em projection property}. This property is needed only for technical reasons, and thus we refer the reader to Section \ref{sec:arcproofs} for details. However, we stress that this restriction is easy to check, and often automatically satisfied, once one is given an explicit subgraph $\G(S)$ of $\Ac(S)$. This is the case with the graphs considered in \cite{AFP} and \cite{Bavard}; see Remark \ref{rem:projection} below.

 
 Before we state our result, recall from \cite{Schleimer} that a {\em witness}\footnote{This definition is due to Schleimer \cite{Schleimer}, who referred to witnesses as {\em holes}. The word ``witness" has been suggested to us by S. Schleimer.} of a subgraph $\G(S)$ of $\mathcal{A}(S)$ is an essential subsurface $Y$ of $S$ such that every vertex of $\G(S)$ intersects $Y$ essentially. Given a witness $Y$, we denote by $\G(Y)$ the subgraph of $\G(S)$ spanned by those vertices of $\G(S)$ that are entirely contained in $Y$. Finally, say that a subgraph of $\Ac(S)$ is {\em sufficiently invariant} if it is invariant under $\Mod(S,P)$,  the $\Mod(S)$-stabiliser of a finite set $P$ of punctures; see Definition \ref{def:invariant}. 
We will prove:

%

\begin{theorem}
Let $S$ be a connected orientable surface of infinite type, and $\G(S)$ a connected, sufficiently invariant subgraph of $\Ac(S)$ with the projection property. 
\begin{enumerate}
\item If every witness of $\G(S)$ has infinitely many punctures, then $\G(S)$ has finite diameter. 
\item Otherwise, $\G(S)$ has infinite diameter. Moreover:
\begin{enumerate}
\item[(2a)] If every two witnesses of $\G(S)$ intersect, then $\G(S)$ is hyperbolic if and only if $\G(Y)$ is uniformly hyperbolic, for every finite-type witness $Y$. 
\item[(2b)] If $\G(S)$ has two disjoint witnesses of finite type, then it is not hyperbolic. 
\end{enumerate}
\end{enumerate}

\label{thm:mainhyp}
\end{theorem}

We stress that part (2b) of Theorem \ref{thm:mainhyp} is merely a manifestation of Schleimer's {\em Disjoint Witnesses Principle} \cite{Schleimer,MS}, although we have included a proof in Section \ref{sec:arcproofs} for completeness. In addition, we remark that once one is given an {\em explicit} subgraph $\G(S)$ of $\Ac(S)$, it is straightforward to decide what the witnesses of $\G(S)$ are and, in particular, where $\G(S)$ falls in the description offered by Theorem \ref{thm:mainhyp}; see the various corollaries below. 
Finally, we will see in Section \ref{sec:arcproofs} that the assumptions that $\G(S)$
 has the projection property  will not be used in the proof of part (1) of Theorem \ref{thm:mainhyp}, and thus that part holds in slightly more generality; this remark will be useful for the various corollaries of Theorem \ref{thm:mainhyp}, see below.

As a special case of Theorem \ref{thm:mainhyp}, we recover the main result of Aramayona-Fossas-Parlier \cite{AFP}; see Section \ref{sec:arcproofs} for the necessary definitions:  

\begin{corollary}[\cite{AFP}]
Let $S$ be a connected orientable surface, and $P$ a non-empty finite set of isolated punctures. Then, the relative arc graph $\Ac(S,P)\subset \Ac(S)$ is hyperbolic and has infinite diameter. 
\label{cor:afp}
\end{corollary}

Once again, we stress that this result was first proved by Bavard \cite{Bavard} in the special case when $S$ is a sphere minus a Cantor set and one isolated puncture. 

We will see in Corollary \ref{cor:Pnonisolated} in Section \ref{sec:arcproofs} that, on the other hand, if $P$ contains a puncture that is not isolated, Theorem \ref{thm:mainhyp} implies that $\Ac(S,P)$ has finite diameter. 
More drastically, if $S$ has no isolated punctures at all, then there are no geometrically interesting $\Mod(S)$-invariant subgraphs of $\Ac(S)$: 

\begin{corollary}
Let $S$ be a connected orientable surface with at least one puncture. If $S$ has no isolated punctures, then any connected $\Mod(S)$-invariant subgraph of $\Ac(S)$ has finite diameter.
\label{cor:cantor} 
\end{corollary}

See Section \ref{sec:arcproofs} for some further consequences of Theorem \ref{thm:mainhyp}.

\subsubsection{Curve graphs} In the light of Corollary \ref{cor:cantor}, there are no geometrically interesting $\Mod(S)$-invariant subgraphs of $\Ac(S)$ if $S$ is a punctured surface with no  isolated punctures. With this motivation we are going to study  $\Mod(S)$-invariant subgraphs of the curve graph $\Cc(S)$ instead.  We will restrict our attention to the case when $S$ has no isolated ends and the space of ends consists solely of planar (resp. non-planar) ends; as we will see, the situation is completely different in these two cases. Before going any further, we note that the case when $S$ has isolated ends is covered in the recent preprint \cite{DFV}; see Remark \ref{rem:dfv} below.

Before we state our results, we denote by $\Nonsep(S)$ the {\em non-separating curve graph} of $S$, namely the subgraph of $\Cc(S)$ spanned by all non-separating curves. Further, let $\Nonsep^*(S)$ be the {\em augmented} nonseparating curve graph of $S$, whose vertices are all nonseparating curves on $S$ together with those curves that cut off a disk containing every puncture of $S$. Finally, denote by $\Outer(S)$ the subgraph of $\Cc(S)$ spanned by all the {\em outer curves} on $S$, namely those curves which cut off a disk containing some, but not all, punctures of $S$. See Section \ref{sec:defcurves} for further definitions.

We start with the case when the genus of $S$ is finite: 

\begin{theorem}
Let $S$ a connected orientable surface of infinite type, with finite genus and no isolated punctures. Then,  a $\Mod(S)$-invariant subgraph $\G(S)\subset \Cc(S)$ has infinite diameter if and only if $\G(S) \cap \Outer(S)  = \emptyset$. Moreover, in this case: 
\begin{enumerate}
\item If  $\G(S) \cap \Nonsep(S)  = \emptyset$ then $\G(S)$ is not hyperbolic.
\item If $\G(S) \cap \Nonsep(S)  \ne  \emptyset$ then $\G(S)$ is quasi-isometric to $\Nonsep(S)$ or $\Nonsep^*(S)$. 
\end{enumerate} 
\label{thm:nonsep}
\end{theorem}

\begin{remark}The classification of infinite-type surfaces (see Theorem \ref{thm:class} in Section \ref{sec:ends}) tells us that, under the hypotheses of the theorem, $S$ is homeomorphic to a closed surface with a Cantor set removed.
\end{remark}

As an immediate consequence of Theorem \ref{thm:nonsep}, we get that if $S$ has genus 0 then any connected, $\Mod(S)$-invariant subgraph of $\Cc(S)$ has finite diameter; compare with Corollary \ref{cor:cantor} above.

In the light of Theorem \ref{thm:nonsep}, a natural problem is the following; compare with Question \ref{qn:q1} below:  

\begin{question}
For $S$ as in Theorem \ref{thm:nonsep}, is $\Nonsep(S)$ (resp. $\Nonsep^*(S)$) hyperbolic? 
\end{question}

As we will see in Proposition \ref{prop:nonsepwitnesses} below, the answer to this question is positive if and only if $\Nonsep(S)$ (resp. $\Nonsep^*(S_{g,n})$) is hyperbolic {\em uniformly} in $n$; compare with part (2a) of Theorem \ref{thm:mainhyp} above. We remark that $\Nonsep(S_{g,n})$ is known to be hyperbolic by the work of Masur-Schleimer \cite{MS} and Hamensd\"adt \cite{Ham}, although the hyperbolicity constant may well depend on $S$. Similary,  $\Nonsep^*(S_{g,n})$ is conjecturally hyperbolic by the work of Masur-Schleimer \cite{MS}, since every two of its witnesses intersect, see Example \ref{example} in Section \ref{sec:defcurves}; on the other hand, even if this were the case, the hyperbolicity constant may well depend on $S$, again.

\medskip

Finally, we deal with the infinite genus surface with a Cantor space of ends, all of them non-planar. We remark that this surface is unique up to homeomorphism, and is usually called the {\em blooming Cantor tree}: 

\begin{theorem}
Let $S$ be a connected orientable surface of infinite genus and no isolated ends. Suppose further that every end of $S$ is non-planar. If $\G(S)$ is a $\Mod(S)$-invariant subgraph of $\Cc(S)$, then $\diam(\G(S)) = 2$. 
\label{thm:curvediam2}
\end{theorem}

\begin{remark} If $S$ has a finite number $\ge 5$ of isolated ends, Durham-Fanoni-Vlamis \cite{DFV} have constructed $\Mod(S)$-invariant subgraphs of $\Cc(S)$ that are hyperbolic and have infinite diameter. 
\label{rem:dfv}
\end{remark}

The plan of the paper is as follows. Section \ref{sec:metricdefs} provides the necessary background on $\delta$-hyperbolic spaces and quasi-isometries. In Section \ref{sec:ends} we recall some facts about the space of ends of a surface. In Section \ref{sec:defcurves} we briefly introduce mapping class groups and  some of the combinatorial complexes one can associate to a surface. In Section \ref{sec:arcproofs} we prove Theorem \ref{thm:mainhyp} and discuss some of its consequences. Finally, Section \ref{sec:curveproofs} contains the proofs of Theorems \ref{thm:nonsep} and  \ref{thm:curvediam2}, together with some open questions.

\medskip

\noindent{\bf Acknowledgements.} This project stemmed out of discussions with Juliette Bavard, and the authors are indebted to her for sharing her ideas and enthusiasm. We want to thank LAISLA  and CONACYT's Red tem\'atica Matem\'aticas y Desarrollo for its support. 

This work started with a visit of the first named author to the UNAM (Morelia), and he would like to thank the Centro de Ciencias Matem\'aticas for its warm hospitality. 
He also thanks Brian Bowditch, Hugo Parlier, and Saul Schleimer for conversations.

The authors are grateful to Federica Fanoni and Nick Vlamis for discussions and for pointing out several errors in an earlier version of this draft.  In the previous version of this paper, Theorem \ref{thm:curvediam2} was not correct, as it omitted the hypothesis of all ends being non-planar; we thank Justin Lanier and Marissa Loving for indicating this to us. Finally, thanks to the referee for comments and suggestions that helped improve the exposition.


\section{Hyperbolic metric spaces}
\label{sec:metricdefs}
We briefly recall some notions on large-scale geometry that will be used in the sequel. For a thorough discussion, see \cite{GH}. 

\begin{definition}[Hyperbolic space]
Let $X$ be a geodesic metric space. We say that $X$ is {\em $\delta$-hyperbolic} if there exists $\delta \ge 0$ such that every triangle $T\subset X$ is {\em $\delta$-thin}: there exists
a point $c\in X$ at distance at most $\delta$ from  every side of $T$. 
\end{definition}

We will simply say that a geodesic metric space is {\em hyperbolic} if it is $\delta$-hyperbolic for some $\delta \ge 0$. 

\begin{definition}[Quasi-isometry]
Let $(X,d_X), (Y,d_Y)$ be two geodesic metric spaces. We say that a map $f:(X,d_X) \to (Y, d_Y)$ is a {\em quasi-isometric embedding} if there exist $\lambda \ge 1$ and $C\ge 0$ such that
\begin{equation}
\frac{1}{\lambda} d_X(x,x') - C \le d_Y(f(x),f(x')) \le \lambda d_X(x,x') + C,
\end{equation}
for all $x,x' \in X$. We say that $f$ is a {\em quasi-isometry} if, in addition to $(1)$, there exists $D\ge 0$ such that $Y$ is contained in the $D$-neighbourhood of $f(X)$. More concretely,  for all $y \in Y$ there exists  $x \in X$ with  $d_Y(y , f(x)) \le D$.
\end{definition}

We say that two  spaces are {\em quasi-isometric} if there exists a quasi-isometry between them. The following is well-known: 

\begin{proposition}
Suppose that two geodesic metric spaces $X,Y$  are quasi-isometric to each other. Then $X$ is hyperbolic if and only if $Y$ is hyperbolic. 
\end{proposition}


\section{The ends of a surface}
\label{sec:ends}

Let $S$ be a connected orientable surface, possibly of infinite topological type. We will briefly recall the definition of the {\em space of ends} of $S$, and refer the reader to \cite{Raymond} and  \cite{Richards} for a more thorough discussion on the space of ends of topological spaces and surfaces respectively.  

\begin{definition}[Exiting sequence]
An {\em exiting sequence} is a collection $U_1 \supseteq U_2 \supseteq \ldots$ of connected open subsets of $S$, such that:
\begin{enumerate}
\item $U_n$ is not relatively compact, for any $n$;
\item The boundary of $U_n$ is compact, for all $n$;
\item Any relatively compact subset of $S$ is disjoint from $U_n$, for all but finitely many $n$. 
\end{enumerate}
\end{definition}

We deem two exiting sequences to be {\em equivalent} if every element of the first sequence is contained in some element of the second, and vice-versa. An {\em end} of $S$ is defined as an equivalence class of exiting sequences, and we write $\Ends(S)$ for the set of ends of $S$. 

Given a subset $U \subset S$ with compact boundary, let $U^*$ be the set of all ends of $S$ that have a representative exiting sequence that is eventually contained in $U$. We put a topology on $\Ends(S)$ by declaring the set of all such $U^*$ to be a basis. 
The following theorem is a special case of Theorem 1.5 in \cite{Raymond}: 

\begin{theorem}
Let $S$ be a connected orientable surface. Then $\Ends(S)$ is totally disconnected, separable, and compact; in particular, it is a subset of a Cantor set. 
\label{thm:ends}
\end{theorem}

We now proceed to describe the classification theorem for connected orientable surfaces of infinite type \cite{Richards}. Before this, we need some notation. Say that an end of $S$ is {\em planar} if it has a representative exiting sequence whose elements are eventually planar; otherwise it is said to be {\em non-planar}. We denote by $\Ends^p(S)$ and $\Ends^n(S)$, respectively, the subspaces of planar and non-planar ends of $S$. Clearly $\Ends(S)=\Ends^p(S)\sqcup\Ends^n(S)$. In \cite{Richards}, Richards proved: 

\begin{theorem}[\cite{Richards}]
Let $S_1$ and $S_2$ be two connected orientable surfaces. Then $S_1$ and $S_2$ are homeomorphic if and only if they have the same genus, and $\Ends^n(S_1)\subset\Ends(S_1)$ is homeomorphic to $\Ends^n(S_2)\subset\Ends(S_2)$ as nested topological spaces. That is, there exists a homeomorphism $h:\Ends(S_1)\to\Ends(S_2)$ whose restriction to $\Ends^n(S_1)$ defines a homeomorphism between $\Ends^n(S_1)$ and $\Ends^n(S_2)$.  
\label{thm:class}
\end{theorem}

We remark that this theorem was later extended by  Prishlyak and Mischenko \cite{PM} to surfaces with non-empty boundary.


\section{Arcs, curves, and witnesses}
\label{sec:defcurves}
In this section we will introduce the necessary definitions about arcs and curves that appear in our results. Throughout, let $S$ be a connected, orientable surface of infinite topological type. Let $\Pi$ be a (possibly empty) set of marked points on $S$, which we feel free to regard as marked points, punctures, or (planar) ends of $S$.  

\subsection{Mapping class group} The mapping class group $\Mod(S)$ is the group of self-homeomorphisms of $S$ that preserve $\Pi$ setwise, up to isotopy preserving $\Pi$ setwise. Given a (possibly empty) finite subset  $P$ of $\Pi$, we define $\Mod(S,P)$ to be the subgroup of $\Mod(S)$ whose every element preserves $P$ setwise. Observe that $\Mod(S,\emptyset) = \Mod(S)$. 

\subsection{Arcs and curves} By a {\em curve} on $S$ we mean the isotopy class of a simple closed curve on $S$ that does not bound a disk with at most one puncture. An {\em arc} on $S$ is the isotopy class of a properly embedded, simple arc on $S$ with both endpoints in $\Pi$. 

Given $a,b\in \AC(S)$, we define their {\em intersection number} as \[i(a,b) = \min \{\bar{a} \cap \bar{b} | \bar{a}\in a, \bar{b}\in b\},\] and say that $a,b$ are {\em disjoint} if $i(a,b)=0$. 
Observe $i(a,b)$ is finite whenever at least one of $a,b$ is a curve, as curves are compact. on the other hand, the intersection number of two arcs could well be infinite; indeed, arcs have infinite length once we equip $S-\Pi$ with a complete hyperbolic structure. 

The {\em arc and curve graph} $\AC(S)$ of $S$ is the simplicial graph whose vertices are all arcs and curves on $S$, and where two vertices are adjacent in $\AC(S)$ if they have disjoint representatives on $S$. As is often the case, we turn $\AC(S)$ into a geodesic metric space by declaring the length of each edge to be 1.

Observe that $\Mod(S)$ acts on $\AC(S)$ by isometries. As mentioned in the introduction, we will concentrate in subgraphs of $\Ac(S)$ that are invariant under big subgroups of $\Mod(S)$. More concretely, we have the following definition:

\begin{definition}[Sufficient invariance]
We say that a subgraph $\G(S)$ of $\AC(S)$ is {\em sufficiently invariant} if there exists a (possibly empty) subset $P$ of $\Pi$ such that $\Mod(S,P)$ acts on $\G(S)$. 
\label{def:invariant}
\end{definition}

We will be interested in various standard $\Mod(S)$-invariant subgraphs of $\AC(S)$, whose definition we now recall.

The {\em arc graph} $\Ac(S)$ is the subgraph of $\AC(S)$ spanned by all vertices of $\AC(S)$ that correspond to arcs on $S$; note that $\Ac(S)=\emptyset$ if and only if $\Pi=\emptyset$. Observe that if $S$ has infinitely many punctures then $\Ac(S)$ has finite diameter. 

Similarly, the {\em curve graph} $\Cc(S)$ is the subgraph spanned by those vertices that correspond to curves on $S$. Note that $\Cc(S)$ has diameter 2 for every surface of infinite type.

The {\em nonseparating curve graph} $\Nonsep(S)$ is the subgraph of $\Cc(S)$ spanned by all nonseparating curves on $S$.
A related graph is the {\em augmented nonseparating curve graph} $\Nonsep^*(S)$, whose vertices are curves that either do not separate $S$, or else bound a disk containing every puncture of $S$. Note that these graphs have diameter 2 if $S$ has infinite genus.

 Finally, the {\em outer curve graph} $\Outer(S)$ is the subgraph of $\Cc(S)$ spanned by those curves $\alpha$ that bound a disk with punctures on $S$, and such that both components of $S-\alpha$ contain at least one puncture of $S$. Observe that $\Outer(S) = \emptyset$ if $S$ is closed or has exactly one puncture, and that $\Outer(S)$ has finite diameter if $S$ has infinitely many punctures. 
 
As the reader may suspect at this point, these observations constitute the main source of inspiration behind the  statements of Theorems \ref{thm:nonsep} and \ref{thm:curvediam2}.

\subsection{Witnesses} Let $S$ be a connected orientable surface of infinite type, and $\G(S)$ a connected subgraph of $\AC(S)$. As mentioned in the introduction, we will use the following notion, originally due to Schleimer \cite{Schleimer}:

\begin{definition}[Witness]
A {\em witness} of $\G(S)$ is an essential subsurface $Y\subset S$ such that every vertex of $\G(S)$ intersects $Y$ essentially. 
\end{definition}

Observe that if $Y$ is a witness of $\G(S)$ and $Z$ is a subsurface of $S$ such that $Y \subset Z$, then $Z$ is also a witness. 

\begin{example} For the sake of concreteness, let $S$ be a connected orientable surface of finite genus $g$, possibly with infinitely many punctures. 
\begin{enumerate}
\item If $\G(S) = \Ac(S)$, then $Y \subset S$ is a witness if and only if $Y$ contains every puncture of $S$. 
\item If $\G(S) = \Cc(S)$, then $Y\subset S$ is a witness if and only if $Y=S$. 
\item If $\G(S) = \Nonsep(S)$, then $Y\subset S$ is a witness if and only if $Y$ has genus $g$.
\item Let $\G(S) = \Nonsep^*(S)$, and suppose $S$ has at least two punctures so that $\Nonsep^*(S) \ne \Nonsep(S)$. 
Then $Y\subset S$ is a witness if and only if $Y$ has genus $g$ and  at least one puncture.
\end{enumerate}
\label{example}
\end{example}


\section{Subgraphs of the arc graph}
\label{sec:arcproofs}

In this section we give a proof of Theorem \ref{thm:mainhyp}. The main tool is the following variant of Masur-Minsky's {\em subsurface projections} \cite{MM2}:


\medskip

\noindent{\bf Subsurface projections.} Let $Y$ be a witness of $\G(S)$, and suppose $Y$ is not homeomorphic to an annulus. There is a natural projection $$\pi_Y:\G(S) \to \Ac(Y)$$ defined by setting $\pi_Y(v)$ to be any connected component of $v \cap Y$. In particular, $\pi_Y(v)=v$ for every $v \subset Y$; in other words, the restriction of $\pi_Y$ to $\G(Y)$ is the identity. Observe that the definition of $\pi_Y$ involves a choice, but any two such choices are disjoint and therefore at distance at most 1 in $\Ac(Y)$. The same argument gives: 

\begin{lemma}
Let $S$ be a surface and $Y$ an essential subsurface not homeomorphic to an annulus. If $u,v$ are disjoint arcs which intersect $Y$ essentially, then $\pi_Y(u)$ and $\pi_Y(v)$ are disjoint (possibly equal).
\label{lem:lip}
\end{lemma}

 Recall that, given a graph $\G(S)\subset \Ac(S)$ and a witness $Y$, by $\G(Y)$ we mean the full subgraph of $\G(S)$ spanned by those vertices that are entirely contained in $Y$. Observe that Lemma \ref{lem:lip}, plus the discussion preceding it, implies that $\G(Y)$ is connected whenever $\G(S)$ is. 
 
   For technical reasons, which will become apparent in the proof of Lemma \ref{lem:qi} below, we will be interested in subgraphs of $\Ac(S)$ for which the subsurface projections defined above satisfy the property described in the following definition:

\begin{definition}[Projection property]
We say that a subgraph $\G(S) \subset \Ac(S)$ has the {\em projection property} if there are constants $\lambda \ge 1$ and $C,D \ge 0$ such that, for every finite-type witness $Y$ of $\G(S)$, the graphs $\pi_Y(\G(S))$ and $\G(Y)$ are quasi-isometric via a $(\lambda,C,D)$-quasi-isometry which restricts to the identity on $\G(Y)$.
\end{definition}

As mentioned in the introduction, we remark that deciding whether a given {\em explicit} subgraph of $\Ac(S)$ has the projection property is normally easy to check; see Remark \ref{rem:projection} below.

The following lemma, which is a small variation of Corollary 4.2 in \cite{AFP}, is the main ingredient in the proof of Theorem \ref{thm:mainhyp}. We note that this is the sole instance in which we will make use of the assumption that $\G(S)$ has the projection property. 

\begin{lemma}
Let $S$ be a surface of infinite type, and $\G(S) \subset \Ac(S)$ a connected subgraph  with the projection property. Then, for every finite-type witness $Y$ of $\G(S)$, the subgraph $\G(Y)$ is uniformly quasi-isometrically embedded in $\G(S)$. 
\label{lem:qi}
\end{lemma}

\begin{proof}
 Let $u,v$ be arbitrary vertices of $\G(Y)$. First, observe that since $\G(Y) \subset \G(S)$, we have $$d_{\G(S)}(u,v) \le d_{\G(Y)}(u,v).$$ To show a reverse coarse inequality, we proceed as follows. Consider a geodesic $\gamma \subset \G(S)$ between $u$ and $v$. The projected  path $\pi_Y(\gamma)$ is a path in $\pi_Y(\G(S))$ between $u= \pi_Y(u)$ and $v=\pi_Y(v)$, and $${\rm length}_{\pi_Y(\G(S))}(\pi_Y(\gamma)) \le {\rm length}_{\G(S)}(\gamma),$$ by Lemma \ref{lem:lip}.
In particular, $$d_{\pi_Y(\G(S))}(u,v) \le d_{\G(S)}(u,v).$$ Since $\G(S)$ has the projection property, there exist constants $L \ge 1$ and $C \ge 0$ (which depend only on $S$) such that 
 $$d_{\G(Y)}(u,v) \le L \cdot d_{\pi_Y(\G(S))}(u,v) + C,$$ and thus the result follows by combining the above two inequalities.
\end{proof}
%

We are now ready to prove Theorem \ref{thm:mainhyp}. 

\begin{proof}[Proof of Theorem \ref{thm:mainhyp}]
 Let $S$ be a connected, orientable surface of infinite  type, and denote by $\Pi$ the set of marked points of $S$. Let $\G(S)$ be a connected subgraph of $\Ac(S)$ with the projection property, and invariant under $\Mod(S, P)$ for some $P \subset \Pi$ finite (possibly empty). 
 
 
 We first prove part (1); in fact, we will show that the diameter of $\G(S)$ is at most 4. Let $u,v$ be two arbitrary distinct vertices of $\G(S)$. We first claim that there exists $w\in \G(S)$ that intersects both $u$ and $v$ a finite number of times. To see this, observe that if $u$ and $v$ have no endpoints in common, then their intersection number is finite and thus we may take $w=u$. Suppose now that $u$ and $v$ share two distinct endpoints $p, p' \in \Pi$. Then there exists an element $h$ in the subgroup of $\Mod(S,P)$ whose every element fixes $p$ and $p'$, such that $w=h(u)$ intersects both $u$ and $v$ a finite number of times, as desired. The rest of cases are dealt with in a similar fashion. This finishes the proof of the claim.

Continuing with the proof, we now claim that there is a vertex $z\in \G(S)$ that  is disjoint from $v$ and $w$. Indeed, consider the surface $F(v,w)$ filled by $v$ and $w$, which has finite type since $v$ and $w$ intersect finitely many times.  Since every witness of $\G(S)$ has infinitely many punctures, by assumption, we deduce that $F(v,w)$ is not a witness, and therefore there exists a vertex $z \in \G(S)$ that does not intersect $F(v,w)$. Using the same reasoning, there exists a vertex $z' \in \G(S)$ that is disjoint from $u$ and $w$. Thus, $$u \to z' \to w \to z \to v$$ is a path of length at most 4 in $\G(S)$ between $u$ and $v$, as desired.  

 \medskip
 
We now proceed to prove part (2), arguing along similar lines to \cite{AFP}. To show that $\G(S)$ has infinite diameter we proceed as follows. By assumption, there exists a witness $Y$ of $\G(S)$ with finitely many punctures. After replacing $Y$ by a finite-type surface containing every puncture of $S$, we may assume that $Y$ has finite type and $\Mod(Y)$ contains a pseudo-Anosov. Essentially by Luo's argument proving that the curve graph of a finite-type surface has infinite diameter (see the comment after Proposition 3.6 of \cite{MM1}), we deduce that $\G(Y)$ has infinite diameter.  Since $\G(Y)$ is quasi-isometrically embedded in $\G(S)$, by Lemma \ref{lem:qi}, it follows that $\G(S)$ has infinite diameter, as desired. 

Next, we establish part (2a) Assume that every two witnesses of $\G(S)$ intersect, and suppose first that there exists $\delta= \delta(S)$ such that $\G(Y)$ is $\delta$-hyperbolic, for every finite-type witness $Y$. We will prove that $\G(S)$ is $\delta$-hyperbolic. To this end, consider  a geodesic triangle $T \subset \G(S)$, and let $Z$ be a witness of $\G(S)$ containing every vertex of $T$, so that $T$ may be viewed as a triangle in $\G(Z)$. First, if $Z$ has infinitely many punctures then $\G(Z)$ has diameter $\le 4$, by the proof of part (1). Therefore $T$ has a $4$-center in $\G(Z)$, and thus also in $\G(S)$, as desired. Assume now that $Z$ has finitely many punctures; in this case, again up to replacing $Z$ by a connected, finite-type surface containing every puncture of $Z$, we may in fact assume that $ Z$ is connected and has finite type. 
 Since $\G(Z)$ is $\delta$-hyperbolic, by assumption, $T$ has a $\delta$-center in $G(Z)$, and thus also in $\G(S)$. Since $T$ is arbitrary and uniformly thin, we obtain that $\G(S)$ is hyperbolic, as claimed.

Using a very similar argument to the one just given, we also deduce that the hyperbolicity of $\G(S)$ implies that of $\G(Y)$, for every finite-type witness $Y$ of $\G(S)$. This finishes the proof of part (2a).

It remains to show part (2b). Assume that  $\G(S)$ has two disjoint witnesses $Y,Z\subset S$, each of finite type. As remarked above, after enlarging $Y$ and/or $Z$ if necessary we may assume that $\G(Y)$ and $\G(Z)$ have infinite  diameter.  
Since $Y$ and $Z$ are witnesses, the projection maps $\pi_Y$ and $\pi_Z$ are well-defined. Therefore there is a projection map $$\pi: \G(S) \to \Ac(Y) \times \Ac(Z)$$ which is simply the map $\pi_Y  \times \pi_Z$. Using this projection and the same arguments as in the proof of Lemma \ref{lem:qi}, the fact that $\G(S)$ has the projection property implies that $\G(S)$ contains a quasi-isometrically embedded copy of $\G(Y) \times \G(Z)$. By choosing a bi-infinite quasi-geodesic in $\G(Y)$ and in $\G(Z)$, we obtain $\G(S)$ contains a quasi-isometrically embedded copy of $\mathbb{Z}^2$, as claimed. 
This finishes the proof of part (2b), and hence of Theorem \ref{thm:mainhyp}. 
\end{proof}

\begin{remark}
As mentioned in the introduction, the proof of part (1) of Theorem \ref{thm:mainhyp} does not use that $\G(S)$ has the projection property; this will be crucial for Corollaries \ref{cor:cantor} and \ref{cor:Pnonisolated} below.
\label{rem:outofjailfree} 
\end{remark}

\subsection{Consequences} We proceed to discuss some of the consequences of Theorem \ref{thm:mainhyp} mentioned in the introduction, starting with Corollary \ref{cor:afp}. Before doing so, we need some definitions from \cite{AFP}. Let $\Pi$ be the set of marked points of $S$, where we assume that $\Pi\ne \emptyset$. As always, we will feel free to view the elements of $\Pi$ as marked points, punctures, or (planar)  ends of $S$. 

 We say that a marked point $p \in \Pi$ is {\em isolated} if it is isolated in $\Pi$, where the latter is equipped with the subspace topology (here we are viewing $\Pi$ as a set of marked points on $S$). Let $P \subset \Pi$ be a non-empty finite subset of marked points on $S$. Define $\Ac(S,P)$ as the subgraph of $\Ac(S)$ spanned by those arcs that have at least one endpoint in $P$. Note that $\Mod(S,P)$ acts on $\Ac(S,P)$, and hence $\Ac(S,P)$ is sufficiently invariant. 

\begin{remark}
The graphs $\Ac(S,P)$ have the projection property: if $Y$ is a finite-type witness of $S$ then $\pi_Y(\Ac(S,P))$ is uniformly quasi-isometric to $\Ac(Y, P)$, which is $\G(Y)$ for $\G(S) = \Ac(S,P)$. The proof that both graphs are quasi-isometric boils down to the fact that, for $v \in \Ac(S,P)$, there is at least one component of $v \cap Y$ that has an endpoint in $P$, which we can use to define a subsurface projection map with nice properties. 
\label{rem:projection}
\end{remark}

We are now in a position to prove Corollary \ref{cor:afp}: 

\begin{proof}[Proof of Corollary \ref{cor:afp}]
Since $P$ is finite and every puncture is isolated, there exists a witness containing only finitely many punctures (any finite-type surface containing $P$ will do). Now, part (2) of Theorem \ref{thm:mainhyp} applies with $\G(S) = \Ac(S,P)$, and thus $\Ac(S,P)$ has infinite diameter. Moreover, if $Y$ is a finite-type witness of $\Ac(S,P)$ then $\G(Y) = \Ac(Y,P)$, which  is $7$-hyperbolic by  \cite{HPW}. 
\end{proof}

We now prove Corollary \ref{cor:cantor}: 

\begin{proof}[Proof of Corollary \ref{cor:cantor}]
Let $S$ be as in the statement, and $\G(S)$ be a connected, $\Mod(S)$-invariant subgraph of $\Ac(S)$. As $S$ has no isolated punctures, using for instance the classification theorem for infinite-type surfaces \cite{Richards} we deduce that there exists an infinite sequence of distinct and pairwise disjoint vertices of $\G(S)$ such that every two distinct arcs in the sequence have no endpoints in common.  In particular, every witness of $\G(S)$ must have an infinite number of punctures, and so the result follows from part (1) of Theorem \ref{thm:mainhyp}. 
\end{proof}

In addition, we recover the following observation due to Bavard (stated as Proposition 3.5 of \cite{AFP}): 

\begin{corollary}
Suppose $P \subset \Pi$  contains a puncture that is not isolated. Then $\Ac(S,P)$ has finite diameter. 
\label{cor:Pnonisolated}
\end{corollary}

Finally, one could define $\Ac(S,P,Q)$ to be, for disjoint finite subsets $P,Q$ of isolated punctures, the subgraph of $\Ac(S)$ spanned by those arcs that have one endpoint in $P$ and the other in $Q$. In this situation we have the following result, also due to Bavard (unpublished): 

\begin{corollary}
The graph $\Ac(S,P,Q)$ is not hyperbolic. 
\end{corollary} 

\begin{proof}
Observe that $Y$ is a witness of  $\Ac(S,P,Q)$ if and only if it contains $P$ or $Q$. In particular, there are two disjoint witnesses of finite type, and part (2b) of Theorem \ref{thm:mainhyp} applies. 
\end{proof}


\section{Subgraphs of the curve graph}
\label{sec:curveproofs}

In this section we deal with connected, $\Mod(S)$-invariant subgraphs of the curve graph, proving Theorems \ref{thm:nonsep} and \ref{thm:curvediam2}. As mentioned in the introduction, we  restrict our attention to the case when $S$ has no isolated ends, which in turn implies that $\Ends(S)$ is homeomorphic to a Cantor set, by the classification theorem for infinite-type surfaces \cite{Richards} described in Section \ref{sec:ends}.

We first prove Theorem \ref{thm:nonsep}. The arguments we will use are similar in spirit to those used in the previous section, but adapted to this particular setting. 

\begin{proof}[Proof of Theorem \ref{thm:nonsep}]
Let $S$ be a connected, orientable surface of infinite type, with finite genus and no isolated ends. Let  $\G(S)$  be a connected, $\Mod(S)$-invariant subgraph of $\Cc(S)$. 

Suppose first that $\G(S) \cap \Outer(S) \ne  \emptyset$.  We want to conclude that $\diam(\G(S)) = 2$. To this end, let $\alpha$ and  $\beta$ be arbitrary vertices of $\G(S)$. If $\alpha$ and $\beta$ are disjoint, there is nothing to prove, so assume that $i(\alpha, \beta) \ne 0$. Let $F(\alpha,\beta)$ be the subsurface of $S$ filled by $\alpha$ and $\beta$, which has finite topological type since $\alpha$ and $\beta$ are compact. Therefore, there exists a connected component $Y$ of $S - F(\alpha, \beta)$ that has infinitely many punctures. Now, the fact that $\Ends(S)$ is a Cantor set and  the classification theorem for infinite-type surfaces, together imply that $\Mod(S)$ acts transitively on $\Outer(S)$. Thus 
there exists $h \in \Mod(S)$ and $\gamma \in \Outer(S)$  such that $h(\gamma) \subset Y$. In particular, $h(\gamma)$ is disjoint from both $\alpha$ and $\beta$ and hence $d_{\G(S)}(\alpha, \beta) =2$. 

Hence from now on, we assume that $\G(S) \cap \Outer(S) = \emptyset$. Suppose first that, in addition, $\G(S) \cap \Nonsep(S) = \emptyset$, and so every element of $\G(S)$ is a curve that either separates $S$ into two surfaces of positive genus, or cuts off a disk containing every puncture of $S$. We claim that $\G(S)$ has two disjoint witnesses, and thus  fails to be hyperbolic. To construct these witnesses, consider a multicurve $M$ consisting of $\rm{genus}(S)+1$ non-separating curves on $S$ such that $S- M= W_1 \sqcup W_2$, with $W_i$ a surface of genus 0 for $i=1,2$ and containing at least one puncture of $S$. By construction, $W_1$ and $W_2$ are witnesses for $\G(S)$. Let $P_i$ be the finite subset of punctures of $W_i$ coming from the elements of $M$. Using subsurface projections as in the previous section gives a quasi-isometric embedding $$\Ac(W_1, P_1) \times \Ac(W_2,P_2) \to \G(S),$$ thus obtaining a quasi-isometrically embedded copy of $\mathbb{Z}^2$ inside $\G(S)$. In particular, $\G(S)$ is not hyperbolic and has infinite diameter. 

Hence, from now on we assume that $\G(S) \cap \Nonsep(S) \ne \emptyset$, which in particular implies that $\Nonsep(S) \subset \G(S)$, since $\Mod(S)$ acts on $\G(S)$. There are two cases to consider: 

\medskip

\noindent{\bf Case I.} {\em No vertex of $\G(S)$ bounds a disk with punctures. }

\medskip

In this case,  we claim:

\medskip

\noindent{\bf Claim.} The inclusion map $\Nonsep(S) \hookrightarrow \G(S)$ is a quasi-isometry.

\begin{proof}[Proof of Claim]
We begin by showing that the inclusion map is a quasi-isometric embedding. In fact, more is true: we will prove that, given $\alpha,\beta\in \Nonsep(S)$ and  a geodesic $\sigma$ in $\G(S)$ between them, we can modify $\sigma$ to a geodesic $\sigma'$ in $\Nonsep(S)$ of the same length. (We remark that this argument is 
contained in the proof that the nonseparating curve complex is connected; see Theorem 4.4 of \cite{FM}.)
Let $\gamma\in \sigma$ be a curve in $\G(S) - \Nonsep(S)$.  By hypothesis, $S - \gamma = Y \cup Z$, where $Y$ and $Z$ both have positive genus. Let $\gamma_L$ and $\gamma_R$ be the vertices of $\sigma$ preceding (resp. following) $\gamma$. The assumption that $\sigma$ is  geodesic implies that either $\gamma_L,\gamma_R \subset Y$ or $\gamma_L,\gamma_R \subset Z$; suppose for the sake of concreteness that we are in the former case. Since $Z$ has positive genus, it contains a nonseparating curve $\gamma'$ which, by construction, is disjoint from $\gamma_L$ and $\gamma_R$. Replacing $\gamma$ by $\gamma'$ on $\sigma$ produces a geodesic in $\G(S)$ with a strictly smaller number of separating curves. 

At this point, we know that the inclusion map $\Nonsep(S) \hookrightarrow \G(S)$ is a (quasi-)isometric embedding. To see that it is a quasi-isometry, observe that every element of $\G(S)$ is at distance at most 1 from an element of $\Nonsep(S)$. This finishes the proof of the claim.
\end{proof}

\noindent{\bf Case II.} {\em There is a vertex of $\G(S)$ which bounds a disk with punctures. }

\medskip

Since $\G(S) \cap \Outer(S) = \emptyset$, we get an inclusion $\Nonsep^*(S) \subset \G(S)$. Using the same arguments as in the previous claim, we obtain: 

\medskip

\noindent{\bf Fact.} The inclusion map $\Nonsep^*(S) \hookrightarrow \G(S)$ is a quasi-isometry.

\medskip

In the light of the claims above, in order to finish the proof of the theorem it suffices to show: 

\medskip

\noindent{\bf Claim.} The graphs $\Nonsep(S)$ and $\Nonsep^*(S)$ have infinite diameter. 

\medskip

\noindent{\em Proof of Claim.} We prove the result for $\Nonsep(S)$, as the case of $\Nonsep^*(S)$  is totally analogous.

In a similar fashion to what we did in the previous section, we are going to prove that, for every finite-type witness $Y$, the subgraph $\Nonsep(Y)$ is quasi-isometrically embedded in $\Nonsep(S)$. Once this has been done, the claim will follow since $\Nonsep(Y)$ has infinite diameter, which again may be deduced using Luo's argument showing that the curve graph has infinite diameter; see Proposition 3.6 of \cite{MM1}.

In this direction, let $Y$ be a finite-type witness of $\Nonsep(S)$; in other words, $Y$ is a finite-type subsurface of $S$ of the same genus as $S$, see Example \ref{example} above. Let $\Ac(Y, \partial Y)$ be the subgraph of $\Ac(Y)$ spanned by those vertices that have both endpoints on $\partial Y$, where $\partial Y$ denotes the boundary of $Y$. 
Similarly, let $\mathcal{A}\Nonsep(Y)$ be the subgraph of $\mathcal{AC}(Y)$ spanned by the vertices of $\Nonsep(Y) \cup \Ac(Y, \partial Y)$. The inclusion map $$\Nonsep(Y)\hookrightarrow \mathcal{A}\Nonsep(Y)$$ is a quasi-isometry, where the constants do not depend on $Y$; to see this, one may use the standard argument to show that the embedding of $C(Y)$ into $\AC(Y)$ is a uniform quasi-isometry.
Now, as in the previous section, there is a subsurface projection $$\pi_Y: \Nonsep(S) \to \mathcal{A}\Nonsep(Y)$$ that associates, to an element of $\Nonsep(S)$, its intersection with $Y$. Using an analogous reasoning to that of Lemma \ref{lem:qi}, we obtain that $\Nonsep(Y)$ is uniformly quasi-isometrically embedded in $\Nonsep(S)$, as desired. This finishes the proof of the claim, and thus that of Theorem \ref{thm:nonsep}.
\end{proof}

The graphs $\Nonsep(S)$ and $\Nonsep^*(S)$ have an intriguing geometric structure. Indeed, using a small variation of the proof of Theorem \ref{thm:mainhyp}, we obtain: 

\begin{proposition}
Let $S$ be a connected surface of finite genus $g$ and with infinitely many punctures. Then $\Nonsep(S)$ (resp. $\Nonsep^*(S)$) is hyperbolic if and only if  $\Nonsep(S_{g,n})$ (resp. $\Nonsep^*(S_{g,n})$) is hyperbolic uniformly in $n$. 
\label{prop:nonsepwitnesses}
\end{proposition}

In the light of Example \ref{example}, the finite-type witnesses of $\Nonsep(S)$ and $\Nonsep^*(S)$ are precisely the subsurfaces of the form $S_{g,n}$; compare with part (3) of Theorem \ref{thm:mainhyp}.

\begin{proof}[Proof of Proposition \ref{prop:nonsepwitnesses}] Again, we argue only for $\Nonsep(S)$, as the other case is very similar. 
Let $T$ be a geodesic triangle in $\Nonsep(S)$. Since $T$ has finitely many vertices and curves are compact, there exists a finite-type subsurface $Y$ of $S$ that contains every element of $T$. Thus we can view $T$ as a geodesic triangle in $\Nonsep(Y)$. If $\Nonsep(S_{g,n})$ is hyperbolic uniformly in $n$, there is $\delta = \delta(g)$ such that $T$ has a $\delta$-center $\alpha \in \Nonsep(Y)$ (with respect to the distance function in $\Nonsep(Y)$). In particular, $\alpha$ is at distance at most $\delta$ from the sides of $T$, where distance is measured in $\Nonsep(Y)$, and hence is a $\delta$-centre for $T$ in $\Nonsep(S)$. Thus, $\Nonsep(S)$ is $\delta$-hyperbolic. 

The other direction is analogous. 
\end{proof}

As mentioned in the introduction, it is known that $\Nonsep(S_{g,n})$ is hyperbolic \cite{Ham,MS}, but in principle the hyperbolicity constant may well depend on $n$. Similarly,  $\Nonsep^*(S_{g,n})$ is conjecturally hyperbolic by Masur-Schleimer's principle that every two witnesses intersect \cite{MS}, but even in this case the hyperbolicity constant could again depend on $n$. Thus we ask: 

\begin{question}
For fixed $g$, are $\Nonsep(S)$ and $\Nonsep^*(S_{g,n})$ hyperbolic uniformly in $n$? More generally, are they hyperbolic uniformly in both $g$ and $n$?
\label{qn:q1}
\end{question}

Finally, we prove Theorem \ref{thm:curvediam2}:

\begin{proof}[Proof of Theorem \ref{thm:curvediam2}]
Let $S$ be a connected orientable surface of infinite genus with no isolated or planar ends; as mentioned above, $S$ is unique up to homeomophism, and is normally referred to as the {\em  blooming Cantor tree}. Consider  a $\Mod(S)$-invariant subgraph  $\G(S)$ of $\Cc(S)$. Let $\alpha$ and $\beta$ be arbitrary vertices of $\G(S)$, noting again that the subsurface $F(\alpha,\beta)$ filled by them has finite type. Choose $h \in \Mod(S)$ such that $h(\alpha) \subset S - F(\alpha, \beta)$. Then  \[\alpha \to h(\alpha) \to \beta\] is  a path of length 2 in $\G(S)$ between $\alpha$ and $\beta$. 
\end{proof}

\bigskip

\noindent Departamento de Matem\'aticas, Universidad Aut\'onoma de Madrid \& \newline \noindent
Instituto de Ciencias Matem\'aticas, CSIC. 
\newline \noindent
\texttt{javier.aramayona@uam.es}

\bigskip

\noindent Centro de Ciencias Matem\'aticas, UNAM (Morelia). \newline \noindent
\texttt{ferran@matmor.unam.mx}

\end{document}